\documentclass{article}
\usepackage{xcolor, hyperref}

\usepackage{fullpage}
\usepackage{amssymb, amsthm}
\usepackage{amsmath}
\usepackage[normalem]{ulem}
\usepackage{bbm}
\usepackage{fancyhdr}
\usepackage{paralist}
\usepackage{comment}
\usepackage{mathtools}
\usepackage{subcaption}
\usepackage{accents}

\usepackage{algorithm}
\usepackage{algpseudocode}

\newcommand{\mexpect}[2]{\mathbb{E}_{#1}\left[ #2\right]}

\newcommand{\norm}[1]{\left\Vert #1 \right\Vert}
\newcommand{\oo}[1]{\left( #1 \right)}

\newcommand{\cc}[1]{\left[ #1 \right]}

\newcommand{\argmin}{\mathop{\arg\min}}

\newtheorem{theorem}{Theorem}
\newtheorem{incorrect_theorem}{Incorrect Theorem}

\title{Differentiable Integer Programming is not Differentiable\\ \& it's not a mere technical problem.}
\author{
Thanawat Sornwanee\thanks{The author thanks Yunzhi Zhang for useful discussions.}\\
\small Stanford University\\
\texttt{tsornwanee@stanford.edu}
}
\date{25 January 2026}

\begin{document}

\maketitle

\begin{abstract}
We show how the differentiability method employed in the paper ``Differentiable Integer Linear Programming'', Geng, et al., 2025~\cite{geng2025differentiable} as shown in its theorem 5 is incorrect. Moreover, there already exists some downstream work that inherits the same error.

The underlying reason comes from that, though being continuous in expectation, the surrogate loss is discontinuous in almost every realization of the randomness, for the stochastic gradient descent.

\end{abstract}

The paper ``Differentiable Integer Linear Programming", Geng, et. al., 2025~\cite{geng2025differentiable} suggests a reformulation of an integer linear program from
\begin{align}
\label{eq:firstproblem}
    \min_{x \in \{0,1\}^d: Ax \preccurlyeq b} c^\top x, 
\end{align}
where $A \in \mathbb{R}^{m \times d}$, $b \in \mathbb{R}^m$, and $c \in \mathbb{R}^d$ for some $m,d \in \mathbb{N}$,
to
\begin{align}
\label{eq:secondproblem}
    \min_{\hat{x} \in [0,1]^d: \sum_{i=1}^m\mexpect{x \sim P_{\hat{x}}}{\max\oo{\left\{a_i^\top x-b_i, 0\right\}}} = 0} c^\top \hat{x},
\end{align}
where the distribution $P_{\hat{x}} := \bigotimes_{i=1}^d \text{Bernoulli}\oo{\hat{x}_i}$ for all $\hat{x} \in [0,1]^d$. This probabilistic reparameterization method is commonly used to relax the integer nature of combinatorial problem~\cite{karalias2020erdos}, while still preserving the correctness of the result. 

The paper itself~\cite{geng2025differentiable} outlines the result into 2 theorems (theorem 1 and theorem 2 of the paper). We summarize the main result of the 2 theorems as the following theorem.
\begin{theorem}
    [Summarization of Theorem 1 and Theorem 2 of~\cite{geng2025differentiable}]
    \begin{align*}
        \argmin_{x \in \{0,1\}^d: Ax \preccurlyeq b} c^\top x = \cc{\argmin_{\hat{x} \in [0,1]^d: \sum_{i=1}^m\mexpect{x \sim P_{\hat{x}}}{\max\oo{\left\{a_i^\top x-b_i, 0\right\}}} = 0} c^\top \hat{x}} \cap \{0,1\}^d
    \end{align*}
\end{theorem}
This suggests that we can solve the problem~\ref{eq:firstproblem} by solving the problem~\ref{eq:secondproblem}.

Note that the searching set of the problem~\ref{eq:secondproblem} is not necessarily convex nor connected, so the paper~\cite{geng2025differentiable} suggests the reformulation from constrained optimization (as in~\ref{eq:secondproblem}) into an unconstrained one: 
\begin{align}
\label{eq:thirdproblem}
    \min_{\hat{x} \in [0,1]^d} c^\top \hat{x} + \mu \sum_{i=1}^m \mexpect{x \sim P_{\hat{x}}}{\max\oo{\left\{a_i^\top x-b_i, 0\right\}}},
\end{align}
and have the theorem 3, which is incorrect. However, the incorrectness of this theorem arises from that the problem~\ref{eq:secondproblem} can be infeasible, and can be considered as a mere technical pathology.\footnote{A detailed explanation on how the theorem 3 of~\cite{geng2025differentiable} is incorrect and how the proof in the paper~\cite{geng2025differentiable} goes wrong is in the appendix~\ref{appendix:theorem3isnotcorrect}, and can likely be attributed to some mathematical typos.}

We can easily correct the theorem as well as adding more detail, making it becomes the following theorem:
\begin{theorem}
    [Corrected Theorem 3 of~\cite{geng2025differentiable}]
    \label{theorem:correct3}
    If there exists some $x \in \{0,1\}^d$ such that $Ax \preccurlyeq b$, then, for any $\mu > \frac{2\sqrt{d} \norm{c}_2}{\min\oo{
            \{1\} \cup 
            \oo{
                \bigcup_{i=1}^m\left\{a_i^\top x - b: x \in \{0,1\}^d\right\}
                \cap (0, \infty)
            }
        }}$, we have that
    \begin{align*}
        \argmin_{\hat{x} \in [0,1]^d} c^\top \hat{x} + \mu \sum_{i=1}^m \mexpect{x \sim P_{\hat{x}}}{\max\oo{\left\{a_i^\top x-b_i, 0\right\}}}
        =
        \argmin_{\hat{x} \in [0,1]^d: \sum_{i=1}^m\mexpect{x \sim P_{\hat{x}}}{\max\oo{\left\{a_i^\top x-b_i, 0\right\}}} = 0} c^\top \hat{x}.
    \end{align*}
\end{theorem}
\begin{proof}
    The proof follows almost from the proof of the theorem 3 of~\cite{geng2025differentiable} in the appendix of~\cite{geng2025differentiable} but with some correction as shown in our appendix~\ref{subsection:incorrectproof3}.
\end{proof}

For a fixed $\mu$, we have that the objective function $c^\top \hat{x} + \mu \sum_{i=1}^m \mexpect{x \sim P_{\hat{x}}}{\max\oo{\left\{a_i^\top x-b_i, 0\right\}}}$ is infinitely differentiable with respect to $\hat{x}$ (in multidimensional sense). However, we can see that the dependence on $\hat{x}$ of the second term goes in as a probability weight over $\{0,1\}^d$. This makes it computational hard to calculate the derivative with respect to $\hat{x}$ if there is no further assumption\footnote{For example, the paper~\cite{karalias2020erdos} considers problem such as max clique and graph-partitioning, allowing the combinatorial cost to only require local structures of the graph solution.} on the sparsity of the matrix $A$.

The paper then uses the reparameterization so that, for each $i \in \{1,2,\dots, m\}$,
\begin{align*}
    \mexpect{x \sim P_{\hat{x}}}{\max\oo{\left\{a_i^\top x - b_i, 0\right\}}}
    &=
    \mexpect{\epsilon \sim \text{Unif}\oo{(0,1)^d}}{\max\oo{\left\{a_i^\top \cc{\mathbf{1}_{\sigma^{-1}\oo{\hat{x}_j} + \sigma^{-1}(\epsilon_j) \ge 0}}_{j=1}^d - b_i, 0\right\}}}\\
    &=
    \mexpect{\epsilon \sim \text{Unif}\oo{(0,1)^d}}{\mathbf{1}_{a_i^\top \cc{\mathbf{1}_{\sigma^{-1}\oo{\hat{x}_j} + \sigma^{-1}(\epsilon_j) \ge 0}}_{j=1}^d - b_i > 0}\oo{a_i^\top \cc{\mathbf{1}_{\sigma^{-1}\oo{\hat{x}_j} + \sigma^{-1}(\epsilon_j) \ge 0}}_{j=1}^d - b_i}}
    ,
\end{align*}
where $\sigma$ is a cumulative distribution function (CDF) of a standard logistic distribution, making $\sigma^{-1}(p) = \log \frac{p}{1-p}$ for all $p \in (0,1)$.
Note that this is correct since the random variable
\begin{align*}
    \cc{\mathbf{1}_{\sigma^{-1}\oo{\hat{x}_j} + \sigma^{-1}(\epsilon_j) \ge 0}}_{j=1}^d \sim P_{\hat{x}}
\end{align*}
when $\epsilon_j \overset{\text{iid}}{\sim} \text{Unif}((0,1))$, which is equivalent to lemma 1 and theorem 4 of the original paper~\cite{geng2025differentiable}. This changes make the probability measure independent from $\hat{x}$ and shifts the dependence on $\hat{x}$ to be only through $\sigma^{-1}$ and $\mathbf{1}_{\cdot \ge 0}$ functions.

Note that, for each realization $\epsilon \in (0,1)^d$, we have that the function\footnote{Note that our $\tilde{\phi}_i\oo{\hat{x}, \epsilon}$ is the same as $\phi_i\oo{\psi\oo{\hat{x}, \epsilon}}$ in the original paper~\cite{geng2025differentiable}.}
\begin{align*}
    \tilde{\phi}_i\oo{\hat{x}, \epsilon} :=
    \mathbf{1}_{a_i^\top \cc{\mathbf{1}_{\sigma^{-1}\oo{\hat{x}_j} + \sigma^{-1}(\epsilon_j) \ge 0}}_{j=1}^d - b_i > 0}\oo{a_i^\top \cc{\mathbf{1}_{\sigma^{-1}\oo{\hat{x}_j} + \sigma^{-1}(\epsilon_j) \ge 0}}_{j=1}^d - b_i}
\end{align*}
is differentiable almost everywhere $\hat{x} \in \text{int}\oo{[0,1]^d} = (0,1)^d$. However, such derivative, if exists, will be $0$, since the step function $\mathbf{1}_{\cdot \ge 0}$ has zero-derivative almost everywhere. Since the function is not a constant function and is, in fact, differentiable, we then have that
\begin{align*}
    \nabla_{\hat{x}} \mexpect{\epsilon \sim \text{Unif}\oo{(0,1)^d}}{\tilde{\phi}_i\oo{\hat{x}, \epsilon}} \ne
    \mexpect{\epsilon \sim \text{Unif}\oo{(0,1)^d}}{\nabla_{\hat{x}} \tilde{\phi}_i\oo{\hat{x}, \epsilon}}.
\end{align*}

The paper~\cite{geng2025differentiable} aims to find a near-optimal and feasible solution to
the integer linear program~\ref{eq:firstproblem} by using a graph neural network (GNN) method for finding a near-optimal solution to the unconstrained problem~\ref{eq:thirdproblem} with some appropriate value of $\mu$. The paper plans to train the GNN using a standard machine learning technique requiring a gradient information (such as stochastic gradient) and backpropagation to update the neural weight parameters in GNN. 

Thus, if we use the stochastic gradient $\nabla_{\hat{x}} \phi_i\oo{\hat{x}, \epsilon}$ in the machine learning training system, then it will be $0$ most of the time making the learning impossible. The paper~\cite{geng2025differentiable} then seeks to find a surrogate loss 
\begin{align*}
    \varphi_i\oo{\hat{x}, \epsilon} \approx \tilde{\phi}_i\oo{\hat{x}, \epsilon}
\end{align*}
such that
\begin{align*}
    \nabla_{\hat{x}} \mexpect{\epsilon \sim \text{Unif}\oo{(0,1)^d}}{\varphi_i\oo{\hat{x}, \epsilon}} =
    \mexpect{\epsilon \sim \text{Unif}\oo{(0,1)^d}}{\nabla_{\hat{x}} \varphi_i\oo{\hat{x}, \epsilon}}.
\end{align*}

\section{Incorrect Attempt in~\cite{geng2025differentiable}}

The paper suggests the use of
\begin{align*}
    \varphi_i\oo{\hat{x}, \epsilon} :=
    \mathbf{1}_{a_i^\top \cc{\mathbf{1}_{\sigma^{-1}\oo{\hat{x}_j} + \sigma^{-1}(\epsilon_j) \ge 0}}_{j=1}^d - b_i > 0}\oo{a_i^\top \cc{\sigma\oo{\sigma^{-1}\oo{\hat{x}_j} + \sigma^{-1}(\epsilon_j)}}_{j=1}^d - b_i}
\end{align*}
in place of
\begin{align*}
    \tilde{\phi}_i\oo{\hat{x}, \epsilon} :=
    \mathbf{1}_{a_i^\top \cc{\mathbf{1}_{\sigma^{-1}\oo{\hat{x}_j} + \sigma^{-1}(\epsilon_j) \ge 0}}_{j=1}^d - b_i > 0}\oo{a_i^\top \cc{\mathbf{1}_{\sigma^{-1}\oo{\hat{x}_j} + \sigma^{-1}(\epsilon_j) \ge 0}}_{j=1}^d - b_i}.
\end{align*}
Recall that $\sigma$ is the CDF of standard logistic distribution, so $\sigma(x) = \frac{1}{1+e^{-x}}$ for all $x \in \mathbb{R}$. It is also easy to see that, similar to $\tilde{\phi}_i$, the function $\varphi_i$ is differentiable for all $\epsilon \in (0,1)^d$ for almost every $\hat{x} \in (0,1)$. The choice of using the CDF $\sigma(\cdot)$ in place of the step function $\mathbf{1}_{\cdot \ge 0}$ is common differentiability technique that one can see in classification literature. This often involves convolution either analytically or empirically~\cite{stewart2023differentiable}\footnote{If it is used in empirical sense, then the gradient has to be done in a monte-carlo manner and ain a finite-difference manner~\cite{bengio2013estimating,stewart2023differentiable}, meaning that we cannot use expected gradient approximates gradient of expectation, since the former is almost surely $0$.} in order for the global structure of the piece-wise function to be passed into a gradient quantity. However, that normally happens without an additional indicator function.\footnote{Which is a feature of a straight-through estimator.}

The paper claims that differentiability almost everywhere is a crucial property allowing   
\begin{align*}
    \nabla_{\hat{x}} \mexpect{\epsilon \sim \text{Unif}\oo{(0,1)^d}}{\varphi_i\oo{\hat{x}, \epsilon}} =
    \mexpect{\epsilon \sim \text{Unif}\oo{(0,1)^d}}{\nabla_{\hat{x}} \varphi_i\oo{\hat{x}, \epsilon}}.
\end{align*}
However, as we have seen earlier, the function $\tilde{\phi}_i$ itself is also differentiable almost everywhere, but does not satisfy the interchangeability between expectation and differentiation. 

The paper provides the following incorrect theorem:
\begin{incorrect_theorem}[Restated Theorem 5 of~\cite{geng2025differentiable}]
    For almost every $\hat{x} \in (0,1)^d$,
    \begin{align*}
        \nabla_{\hat{x}} \mexpect{\epsilon \sim \text{Unif}\oo{(0,1)^d}}{\varphi_i\oo{\hat{x}, \epsilon}}
        =
        \mexpect{\epsilon \sim \text{Unif}\oo{(0,1)^d}}{
        \mathbf{1}_{a_i^\top \cc{\mathbf{1}_{\sigma^{-1}\oo{\hat{x}_j} + \sigma^{-1}(\epsilon_j) \ge 0}}_{j=1}^d - b_i > 0}
        \oo{a \odot \cc{\frac{\partial}{\partial x_j}\sigma\oo{\sigma^{-1}\oo{\hat{x}_j} + \sigma^{-1}(\epsilon_j)}}_{j=1}^d}
        },
    \end{align*}
    where $\odot$ denotes elementwise multiplication.
\end{incorrect_theorem}

This theorem primarily suggests that we can interchange the expectation and differentiation. We can see that, if we can do so, then, for each realization of $\epsilon$, we can compute the differentiation term of each dimension in a parallel manner before multiply it by the indicator $\mathbf{1}_{a_i^\top \cc{\mathbf{1}_{\sigma^{-1}\oo{\hat{x}_j} + \sigma^{-1}(\epsilon_j) \ge 0}}_{j=1}^d - b_i > 0}$ to see whether the $i^{\text{th}}$ constraint is violated. 

Unlike the incorrectness in the theorem 3 of~\cite{geng2025differentiable}, the incorrectness of the theorem 5 of~\cite{geng2025differentiable} is not merely a technical problem. This happens primarily because the function $\varphi$ itself is not continuous.\footnote{Although continuity and differentiability almost everywhere may not be a sufficient condition, other additional conditions can be thought of as a technical conditions (such as absolute continuity and difference quotient domination).}

\subsection{Why Continuity Matters: a Minimal Example}

    As a simple example, we can consider a 1-dimensional case when the function of interest is
    \begin{align*}
        f\oo{\hat{x}, \epsilon} := \mathbf{1}_{\epsilon \le \hat{x}} \hat{x},
    \end{align*}
    making $\frac{\partial}{\partial \hat{x}} f\oo{\hat{x}, \epsilon} = \mathbf{1}_{\epsilon \le \hat{x}}$ for any $\hat{x} \in (0,1) - \{\epsilon\}$.
    
    This makes
    \begin{align*}
        \mexpect{\epsilon \sim \text{Unif}((0,1))}{\frac{\partial}{\partial \hat{x}} f\oo{\hat{x}, \epsilon}} = 
        \mexpect{\epsilon \sim \text{Unif}((0,1))}{\mathbf{1}_{\epsilon \le \hat{x}}}
        =
        \int_{\epsilon=0}^{\hat{x}} 1 d\epsilon = \hat{x},
    \end{align*}
    while
    \begin{align*}
        \frac{d}{d\hat{x}}\mexpect{\epsilon \sim \text{Unif}((0,1))}{f\oo{\hat{x}, \epsilon}}
        =
        \frac{d}{d\hat{x}}\mexpect{\epsilon \sim \text{Unif}((0,1))}{\mathbf{1}_{\epsilon \le \hat{x}} \hat{x}}
        =
        \frac{d}{d\hat{x}}
        \cc{
        \int_{\epsilon  =0 }^{\hat{x}}
        \hat{x} d\epsilon
        }
        =
        2\hat{x}.
    \end{align*}
    
    It can be seen that the discrepancy in this 1-dimensional comes from the change of the integral boundary, but this can also be thought of as a result of discontinuity of the original function, which has led us to use integration with boundary dependent on $\hat{x}$ in the first place.
    
    In this simple example, we can see that the derivative can be fixed by adding a boundary term so that
    \begin{align*}
        \frac{d}{d\hat{x}}\mexpect{\epsilon \sim \text{Unif}((0,1))}{f\oo{\hat{x}, \epsilon}}
        =
        \mexpect{\epsilon \sim \text{Unif}((0,1))}{\frac{\partial}{\partial \hat{x}} f\oo{\hat{x}, \epsilon}}
        + \oo{\frac{d}{d\hat{x}}\hat{x}} \left[\epsilon\middle\vert_\epsilon=\hat{x}\right].
    \end{align*}
    
    However, this relies on the fact that the boundary term is easy to find in this case. For the function $\varphi_i$ used in~\cite{geng2025differentiable}, there is no simple correction without further knowledge on the structure of $A$ and $b$. Thus, we cannot fix it in a dimension scalable manner.

\section{1-Dimensional Example}

We consider a concrete example with 1-dimension, so we drop the subscript. We consider the case when $a>b>0$, so the only feasible solution to the combinatorial problem~\ref{eq:firstproblem} is $x = 0$. Similarly, we have that the only feasible solution to the problem~\ref{eq:secondproblem} is $\hat{x} = 0$.

For simplicity, we consider the case where $c = 0$, so, by theorem~\ref{theorem:correct3}, we can choose $\mu$ to be $1$\footnote{However, the behavior will be similar to the case when $c$ is non-zero but $\mu$ is sufficiently large, since only the ratio  between $c$ and $\mu$ matters.} 

We have that, for any $\hat{x} \in (0,1)$,
\begin{align*}
    \hat{\varphi}\oo{\hat{x}}
    :&=
    \mexpect{\epsilon \sim \text{Unif}((0,1))}{\varphi\oo{\hat{x}, \epsilon}}\\
    &=
    \int_{\epsilon =0}^1 \cc{
        a\sigma\oo{\sigma^{-1}\oo{\hat{x}} + \sigma^{-1}\oo{\epsilon}}-b
    }
    \mathbf{1}_{a\oo{
    \mathbf{1}_{\sigma^{-1}\oo{\hat{x}} + \sigma^{-1}\oo{\epsilon} \ge 0}
    }-b > 0}
    d\epsilon\\
    &=
    \int_{\epsilon =0}^1 \cc{
        a\sigma\oo{\sigma^{-1}\oo{\hat{x}} + \sigma^{-1}\oo{\epsilon}}-b
    }
    \mathbf{1}_{\epsilon > 1-\hat{x}}
    d\epsilon
    \\
    &=
    \int_{\epsilon =0}^1 \cc{
        -a\frac{1}{1+e^{\sigma^{-1}\oo{\hat{x}} + \sigma^{-1}\oo{\epsilon}}}+ (a-b)
    }
    \mathbf{1}_{\epsilon > 1-\hat{x}}
    d\epsilon
    \\
    &=
    \int_{\epsilon =0}^1 \cc{
        -a\frac{1}{1+\frac{\hat{x}}{1-\hat{x}} \frac{\epsilon}{1-\epsilon}}+ (a-b)
    }
    \mathbf{1}_{\epsilon > 1-\hat{x}}
    d\epsilon
    \\
    &=
    a\oo{-\int_{\epsilon =1-\hat{x}}^1
        \frac{1}{1+\frac{\hat{x}}{1-\hat{x}} \frac{\epsilon}{1-\epsilon}} d\epsilon}
    + (a-b)\hat{x}\\
    &=
    a f_1\oo{\hat{x}}
    + (a-b)\hat{x},
\end{align*}
where
\begin{align*}
    f_1\oo{\hat{x}}
    :=
    -\int_{\epsilon =1-\hat{x}}^1
        \frac{1}{1+\frac{\hat{x}}{1-\hat{x}} \frac{\epsilon}{1-\epsilon}} d\epsilon
    =
    \begin{cases}
        \frac{\hat{x} \oo{1-\hat{x}}}{\oo{2\hat{x}-1}^2}\cc{\ln \oo{2-2\hat{x}} + 2 \hat{x}-1} &\text{ if }\hat{x} \ne \frac{1}{2}\\
        -\frac{1}{8}  &\text{ if }\hat{x} = \frac{1}{2}
    \end{cases}.
\end{align*}
Although $f_1$ is by definition has a domain being $(0,1)$, we can see that we can naturally extend it to $[0,1]$ from continuity, and have that $f_1(0)=f_1(1)=0$. Furthermore, $f_1$ is bounded, non-positive, and strictly convex in $\hat{x} \in [0,1]$. Thus, the function $\hat{\varphi}$ is also bounded and strictly convex, thereby being closed and proper as well as having a unique minimizer~\cite{ryu2022large}. However, since the unique minimizer of $f_1$ is greater than $0.5$, with sufficiently small $\frac{b}{a-b}$, we will have that the unique minimizer of $\hat{\varphi}$ will also be greater than $0.5$.\footnote{See figure~\ref{fig:potential}.} Thus, the unique minimizer $\hat{x}^*$ of $\hat{\varphi}$ will more than half of the time yields an invalid $x$ from the distribution $P_{\hat{x}^*}$. Note that, whenever $a>b$, the value of $\hat{\varphi}$ will exceed $0$ when $\hat{x}$ is sufficiently big, but it is possible for $\hat{\varphi}$ to have a negative value in the interior.

\begin{figure}
    \centering
    \includegraphics[width=0.7\textwidth]{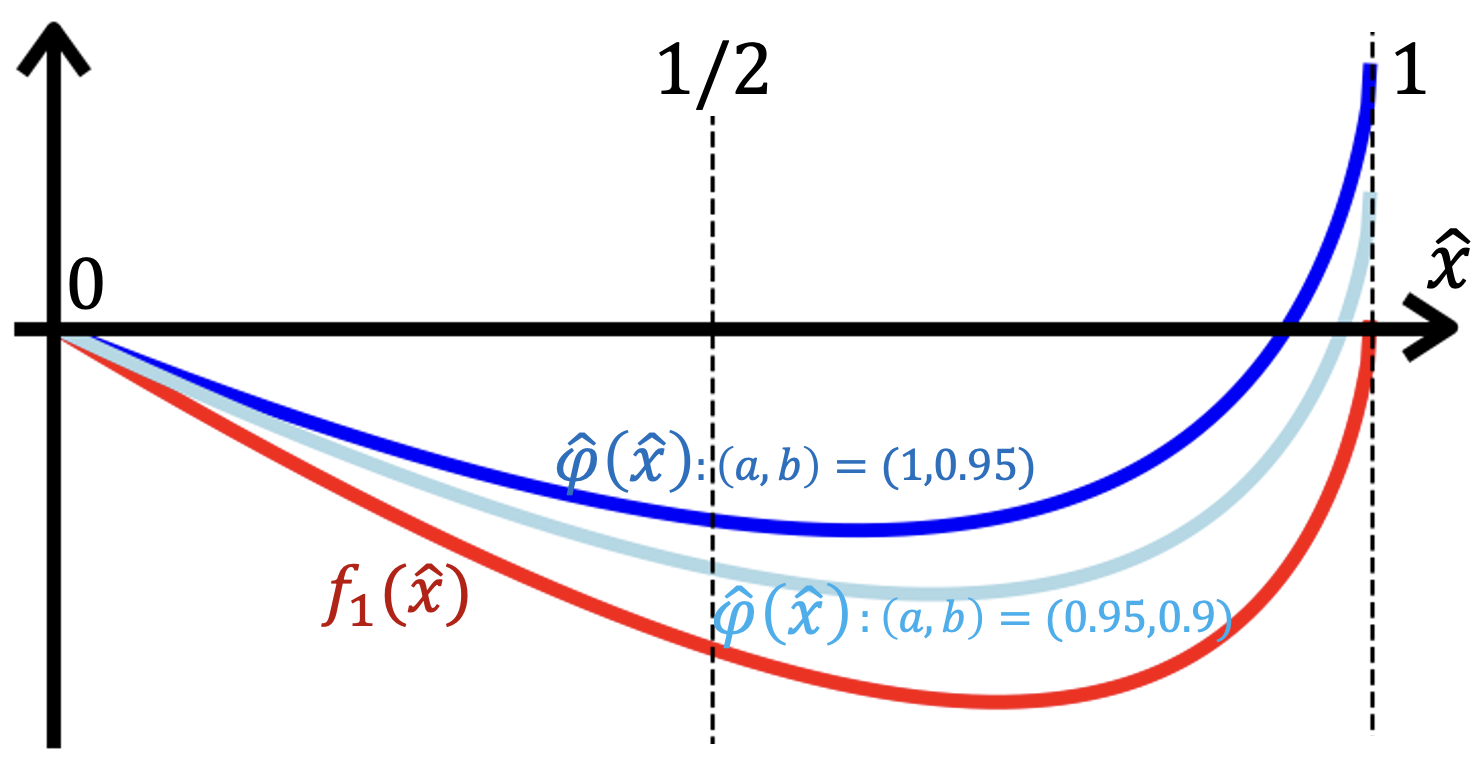}
    \caption{In 1-dimensional case where $a>b>0$, we will have that the loss proposed in~\cite{geng2025differentiable} is $\hat{\varphi}$, which is a transformation of the function $f_1$, which is independent from $a$ and $b$. Both functions are bounded and convex, so the minimizer is unique. Strict convexity suggests that the unique minimizer of $\hat{\varphi}$ will be close to that of $f_1$ when $\frac{b}{a-b}$ is close to $0$, so such minimizer will still be greater than $0.5$.}
    \label{fig:potential}
\end{figure}

\subsection{How the Algorithm in~\cite{geng2025differentiable} Works}

    Recall that
    \begin{align}
    \label{eq:4}
        \hat{\varphi}\oo{\hat{x}} =
        a\oo{-\int_{\epsilon =1-\hat{x}}^1
            \frac{1}{1+\frac{\hat{x}}{1-\hat{x}} \frac{\epsilon}{1-\epsilon}} d\epsilon}
        + (a-b)\hat{x},
    \end{align}
    while the expected stochastic gradient term is
    \begin{align*}
        \hat{g}\oo{\hat{x}}
        :&=
        \mexpect{\epsilon \sim \text{Unif}\oo{(0,1)^d}}{
        \mathbf{1}_{a_i^\top \cc{\mathbf{1}_{\sigma^{-1}\oo{\hat{x}_j} + \sigma^{-1}(\epsilon_j) \ge 0}}_{j=1}^d - b_i > 0}
        \oo{a \odot \cc{\frac{\partial}{\partial x_j}\sigma\oo{\sigma^{-1}\oo{\hat{x}_j} + \sigma^{-1}(\epsilon_j)}}_{j=1}^d}
        }\\
        &=
        a\oo{-\int_{\epsilon =1-\hat{x}}^1 \frac{\partial}{\partial\hat{x}}
            \frac{1}{1+\frac{\hat{x}}{1-\hat{x}} \frac{\epsilon}{1-\epsilon}} d\epsilon}
        + (a-b),
    \end{align*}
    so we can see that the difference between $\hat{g}\oo{\hat{x}}$ and $\frac{d}{d\hat{x}}\hat{\varphi}\oo{\hat{x}}$ is from the boundary term. Formally, we note that
    \begin{align*}
        \frac{d}{d\hat{x}}\hat{\varphi}\oo{\hat{x}}
        &=
        a\oo{-\int_{\epsilon =1-\hat{x}}^1 \frac{\partial}{\partial\hat{x}}
            \frac{1}{1+\frac{\hat{x}}{1-\hat{x}} \frac{\epsilon}{1-\epsilon}} d\epsilon}
        +a \frac{d\oo{1-\hat{x}}}{d\hat{x}} \frac{1}{1+\frac{\hat{x}}{1-\hat{x}} \frac{1-\hat{x}}{1-\oo{1-\hat{x}}}}
        + (a-b)\\
        &=
        \hat{g}
        a\oo{-\int_{\epsilon =1-\hat{x}}^1 \frac{\partial}{\partial\hat{x}}
            \frac{1}{1+\frac{\hat{x}}{1-\hat{x}} \frac{\epsilon}{1-\epsilon}} d\epsilon}
        -\frac{a}{2}
        + (a-b),
    \end{align*}
    making
    \begin{align*}
        \hat{g}\oo{\hat{x}} = \frac{a}{2}+\frac{d}{d\hat{x}}\hat{\varphi}\oo{\hat{x}}
        =
        a\oo{f_1'\oo{\hat{x}}+\frac{1}{2}} + (a-b).
    \end{align*}
    We can see that $\min_{\hat{x} \in [0,1]}\oo{f_1'\oo{\hat{x}}+\frac{1}{2}} = f_1'(0)+\frac{1}{2} = \ln 2 > 0$. Therefore, $\hat{g}\oo{\hat{x}}$ is uniformly bounded away from $0$, since, for any value of $\hat{x} \in [0,1]$, it remains $\ge a \ln 2 + a-b>0$. Thus, the algorithm will always have its stochastic gradient descent (in expectation) always move the point $x$ towards $0$.

    Specifically,
    it will be as if we perform a stochastic gradient descent over the function
    \begin{align*}
        \hat{\hat{\varphi}}\oo{\hat{x}} := a \hat{f}\oo{\hat{x}} + (a-b)\oo{\hat{x}},
    \end{align*}
    where
    \begin{align*}
        \hat{f}\oo{\hat{x}} := f\oo{\hat{x}} +\frac{\hat{x}}{2},
    \end{align*}
    because
    \begin{align*}
        \hat{\hat{\varphi}}'\oo{\hat{x}} = \hat{g}\oo{\hat{x}}.
    \end{align*}
    This means that the algorithm in~\cite{geng2025differentiable} actually performs stochastic gradient descent on the new potential $\hat{\hat{\varphi}}\oo{\hat{x}} = \hat{\varphi}\oo{\hat{x}} + \frac{a\hat{x}}{2}$, which is strictly increasing on $[0,1]$, instead of on $\hat{\varphi}_1\oo{\hat{x}}$ as claimed in its theorem 5.\footnote{However, at least in 1 dimensional case, the new potential turns out to be much better than the former one, since $\hat{g}\oo{\hat{x}}$ is uniformly bounded away from $0$, so this may explain why the algorithm in~\cite{geng2025differentiable} can still achieve a good result. However, we can see in the section~\ref{section: 2d} that the uniform bound does not apply to 2 dimensional case.} The difference between the new potential $\hat{\hat{\varphi}}$, the old potential $\hat{\varphi}$, and the actual loss of interest $\hat{\phi}\oo{\hat{x}} := \mexpect{\epsilon \sim \text{Unif}((0,1))}{\phi\oo{\hat{x}, \epsilon}}$ can be seen in the figure~\ref{figure:2}. 
    
    This is a common problem for an optimization system when discontinuity happens either from the nature of the loss function or from underlying decision system~\cite{jalota2024simple, sornwanee20251}, making the local direction possibly lead to worse outcome. In this case, the discontinuity comes from the multiplier $\mathbf{1}_{z \ge 0}$. It is then natural that the optimal $z$ is to allow a considerable mass of the yellow distribution to go to the positive region in order to get the negative loss.

    \begin{figure}[p]
        \centering
        \begin{subfigure}{0.7\textwidth}
            \centering
            \includegraphics[width=\linewidth]{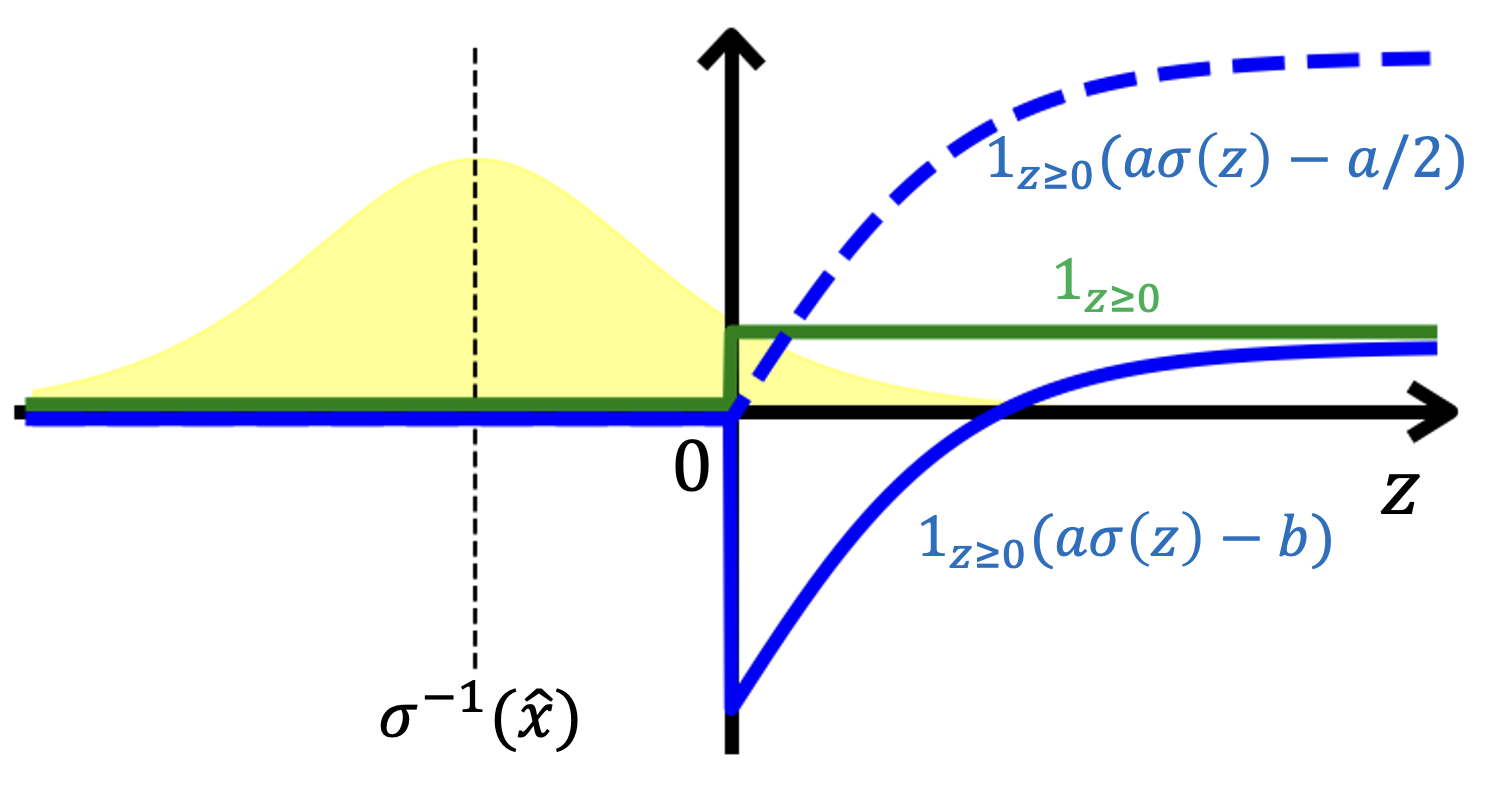}
            \caption{The Convolution Form}
        \end{subfigure}    
        \begin{subfigure}{0.7\textwidth}
            \centering
            \includegraphics[width=\linewidth]{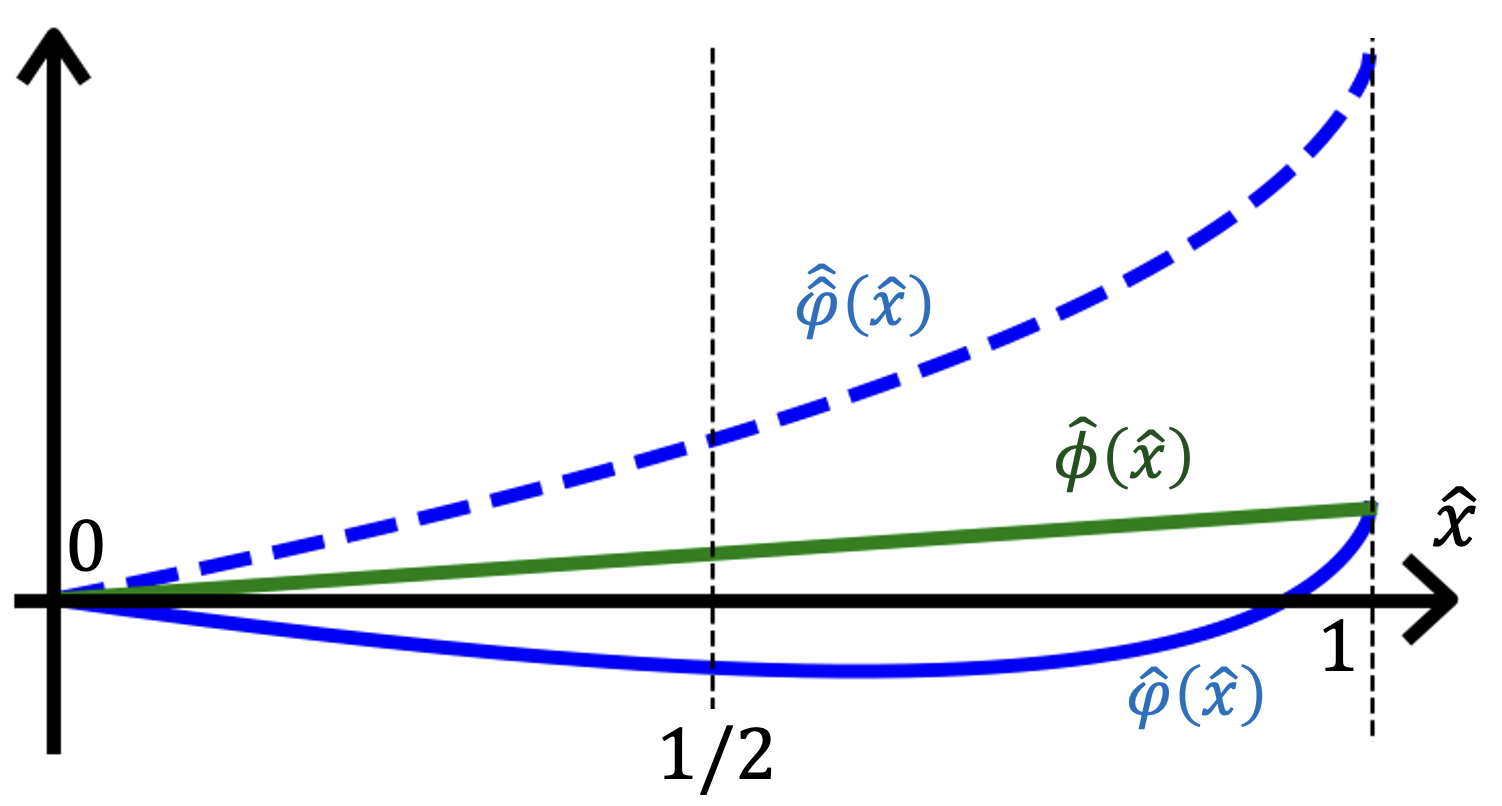}
            \caption{The Potential Form}
        \end{subfigure}
        \caption{In these two figures, we set $(a,b) = (1, 0.95)$. To evaluate the actual loss $\hat{\phi}\oo{\hat{x}}$, we start from encoding $\hat{x}$ to be $z = \sigma^{-1}\oo{\hat{x}}$, and sample a logistic distribution to be added to $\sigma^{-1}\oo{\hat{x}}$ (as shown as the yellow density in the top figure). It is easy to see that $\hat{\phi}\oo{\hat{x}}$ will be the integration of the step function (in green) with respect to the density (in yellow).\\The paper~\cite{geng2025differentiable} suggests a surrogate function (as the bold blue line), which is $0$ when lower than $0$ and converge to $1$ when $z \to \infty$. The integration of such function with respect to the yellow density will yield $\hat{\varphi}\oo{\hat{x}}$.\\When the value of the encoded $z = \sigma^{-1}\oo{\hat{x}}$ shifts, the yellow density will also shift by the same distance. Thus, if a function is continuous together with some regularity conditions, we will have that the derivative (with respect to $z$) of the integration of it with respect to the yellow density will be the same as the integration of its derivative (with respect to $z$) with respect to the yellow density. However, the bold blue function is discontinuous, so such equality does not hold.\\We note that the dotted blue line is the bold blue function with an upward shift for the positive region of $b - \frac{a}{2} > 0$. Now, we have that the dotted blue function is absolutely continuous and has the same derivative as the bold blue function almost everywhere. Thus, we have that the algorithm in~\cite{geng2025differentiable} would not be changed if the bold blue function were to be the dotted one. For the dotted function, we can swap the integration and differentiation, so the algorithm will be as if we perform stochastic gradient on $\hat{\hat{\varphi}}$, which is the integral of the dotted blue line, and is different from the proposed surrogate $\hat{\varphi}$.
        }
        \label{figure:2}
    \end{figure}

    Moreover, we will also have that, the algorithm may not even implement any stochastic gradient descent in a more than one dimensional case, meaning that there does not exists any absolutely continuous\footnote{We restrict our attention to the potential that is absolutely continuous to ensure that gradient descent will reduce the potential. This is commonly assumed and used.} potential $\hat{\hat{\varphi}}$ such that $\hat{g} = \nabla \hat{\hat{\varphi}}$. The example will be shown in the 2 dimensional case in the section~\ref{section: 2d}.

    \subsection{Possible Fix?}

    If we want to commit to the surrogate loss (the old potential) of 
    \begin{align*}
        \hat{\varphi}\oo{\hat{x}} := \mexpect{\epsilon \sim \text{Unif}((0,1))}{\varphi\oo{\hat{x}, \epsilon}}
    \end{align*}
    as a good plug-in approximation of the actual loss of
    \begin{align*}
        \hat{\phi}\oo{\hat{x}} := \mexpect{\epsilon \sim \text{Unif}((0,1))}{\phi\oo{\hat{x}, \epsilon}}.
    \end{align*}
    from that the new potential $\tilde{\hat{\varphi}}_1$ differs from the former potential $\hat{\varphi}_1$ by $\frac{ax}{2}$, it is tempting to think that we can easily fix the stochastic gradient in~\cite{geng2025differentiable} by adding a constant term. However, there are main problems with that. First, even in 1 dimensional case, we requires access to the global structure to determine whether we have to add such constant term of $\frac{a}{2}$. For example, if $b > a$ or $b <0$, then the integral in the equation~\ref{eq:4} has its boundaries being independent from the value of $\hat{x}$. Thus, the constant term, which comes from that $\frac{1}{1+\frac{\hat{x}}{1-\hat{x}}\frac{\epsilon}{1-\epsilon}} = \frac{1}{2}$ at the boundary of $\epsilon = 1-\hat{x}$, will not be applied. Note that this is easy to fix in 1-dimensional case\footnote{However, in 1-dimensional case, it is also easy to solve the integer linear programming directly as a simple function of $a$ and $b$}, but in high dimensional regime, we need to know the behavior of the constraint when we change the assignment. This computation tends to be difficult in a general integer linear programming.

    The second problem is that having the function being a constant at the boundary is easily achieved in 1 dimension but not in higher dimension as can be seen in the example with 2 dimension in the section~\ref{section: 2d}.

\section{2-Dimensional Example}
\label{section: 2d}

    We have seen that, in a 1-Dimensional case, we can still rationalize the algorithm in~\cite{geng2025differentiable} as a stochastic gradient descent on a new potential $\hat{\hat{\varphi}}$ instead of the surrogate loss $\hat{\varphi}$ proposed and claimed in the paper. In a 2 dimensional case, we will see that it is possible that $\hat{g}$ cannot be rationalized as a gradient of any absolutely continuous potential $\hat{\hat{\varphi}}$. We will show this by claiming that the curl
    \begin{align*}
        \vec{\nabla} \times \hat{g} \ne 0.
    \end{align*}

    We will restrict the attention when there is only a single constraint in the integer linear programming~\ref{eq:firstproblem}, so the subscript will be omitted. Thus, the problem will be specified by $a$, $b$, and $c$.
    
    We will overload the notation and denote $a = (a_1, a_2)$, and assume that $c = (0,0)$. Thus, the problem will be specified by $a_1$, $a_2$, and $b$.
    
    Consider the case when $a_1+a_2 > b \ge a_1 \ge a_2 > 0$, so the optimal solutions for the original problem~\ref{eq:firstproblem} are $(0,0)$, $(0,1)$, and $(1,0)$. The set of the optimal solutions to the reformulated problem~\ref{eq:thirdproblem} is the L-shape set
    \begin{align*}
        \oo{[0,1] \times \{0\}} \cup \oo{\{0\} \times [0,1]}.
    \end{align*}

    We, similarly, have that
    \begin{align*}
        \hat{\varphi}\oo{\hat{x}}
        :&=
        \mexpect{\epsilon \sim \text{Unif}\oo{(0,1)^2}}{\varphi\oo{\hat{x}, \epsilon}}\\
        &=
        \int_{\epsilon_1 =0}^1\int_{\epsilon_2 =0}^1 \cc{
            \sum_{i=1}^2a_i\sigma\oo{\sigma^{-1}\oo{\hat{x}_i} + \sigma^{-1}\oo{\epsilon_i}}-b
        }
        \mathbf{1}_{\sum_{i=1}^2a_i\oo{\mathbf{1}_{\sigma\oo{\sigma^{-1}\oo{\hat{x}_i} + \sigma^{-1}\oo{\epsilon_i}} > 0}}-b > 0}
        d\epsilon_2 d\epsilon_1\\
        &=
        \int_{\epsilon_1 =0}^1\int_{\epsilon_2 =0}^1 \cc{
            \sum_{i=1}^2a_i\sigma\oo{\sigma^{-1}\oo{\hat{x}_i} + \sigma^{-1}\oo{\epsilon_i}}-b
        }
        \mathbf{1}_{\epsilon_1 >1 - \hat{x}_1}
        \mathbf{1}_{\epsilon_2 >1 - \hat{x}_2}
        d\epsilon_2 d\epsilon_1\\
        &=
        \int_{\epsilon_1 =1 - \hat{x}_1}^1\int_{\epsilon_2 =1 - \hat{x}_2}^1 \cc{
            \sum_{i=1}^2a_i\sigma\oo{\sigma^{-1}\oo{\hat{x}_i} + \sigma^{-1}\oo{\epsilon_i}}-b
        }
        d\epsilon_2 d\epsilon_1\\
        &=
        \sum_{i=1}^2
        \int_{\epsilon_1 =1 - \hat{x}_1}^1\int_{\epsilon_2 =1 - \hat{x}_2}^1 \cc{
            a_i\sigma\oo{\sigma^{-1}\oo{\hat{x}_i} + \sigma^{-1}\oo{\epsilon_i}}-b
        }
        d\epsilon_2 d\epsilon_1\\
        &=
        -b\hat{x}_1\hat{x}_2+
        \hat{x}_1
        \int_{\epsilon_2 =1 - \hat{x}_2}^1 a_2\sigma\oo{\sigma^{-1}\oo{\hat{x}_2} + \sigma^{-1}\oo{\epsilon_2}} - \frac{b}{2} d\epsilon_2 
        +
        \hat{x}_2
        \int_{\epsilon_1 =1 - \hat{x}_1}^1 a_1\sigma\oo{\sigma^{-1}\oo{\hat{x}_1} + \sigma^{-1}\oo{\epsilon_1}} - \frac{b}{2} d\epsilon_1 \\
        &=
        \hat{x}_2a_1 \cc{f_1\oo{\hat{x}_1} + \hat{x}_1} + 
        \hat{x}_1a_2 \cc{f_1\oo{\hat{x}_2} + \hat{x}_2} - b \hat{x}_1\hat{x}_2\\
        &=
        \hat{x}_2a_1f_1\oo{\hat{x}_1} + 
        \hat{x}_1a_2 f_1\oo{\hat{x}_2} + \oo{a_1+a_2- b} \hat{x}_1\hat{x}_2,
    \end{align*}
    where $f_1$ is as defined in the 1-dimensional example section.

    Similar to the behavior seen in 1-dimensional, we can see that, if $a_1+a_2$ is sufficiently close to $b$, then the minimizer of $\hat{\varphi}\oo{\hat{x}}$ is in the interior of $[0,1]^2$. 
    
    Moreover, the expectation of the supposedly stochastic gradient in~\cite{geng2025differentiable} is
    \begin{align*}
        \hat{g}\oo{\hat{x}}
        :&=
        \mexpect{\epsilon \sim \text{Unif}\oo{(0,1)^2}}{\nabla_{\hat{x}}\varphi\oo{\hat{x}, \epsilon}}\\
        &=
        \int_{\epsilon_1 =0}^1\int_{\epsilon_2 =0}^1 \nabla_{\hat{x}}\cc{
            \sum_{i=1}^2a_i\sigma\oo{\sigma^{-1}\oo{\hat{x}_i} + \sigma^{-1}\oo{\epsilon_i}}-b
        }
        \mathbf{1}_{\sum_{i=1}^2a_i\oo{\mathbf{1}_{\sigma\oo{\sigma^{-1}\oo{\hat{x}_i} + \sigma^{-1}\oo{\epsilon_i}} > 0}}-b > 0}
        d\epsilon_2 d\epsilon_1\\
        &=
        \int_{\epsilon_1 =1-\hat{x}_1}^1\int_{\epsilon_2 =1-\hat{x}_2}^1 \nabla_{\hat{x}}\cc{
            \sum_{i=1}^2a_i\sigma\oo{\sigma^{-1}\oo{\hat{x}_i} + \sigma^{-1}\oo{\epsilon_i}}-b
        }
        d\epsilon_2 d\epsilon_1,
    \end{align*}
    making the first coordinate
    \begin{align*}
        \cc{\hat{g}\oo{\hat{x}}}_1
        =
        \oo{\int_{\epsilon_2 =1-\hat{x}_2}^1 d\epsilon_2}\oo{
        \int_{\epsilon_1 =1-\hat{x}_1}^1 \frac{\partial}{\partial{\hat{x}_1}}\cc{
            a_1\sigma\oo{\sigma^{-1}\oo{\hat{x}_1} + \sigma^{-1}\oo{\epsilon_1}}-b
        }
         d\epsilon_1}
         = \hat{x}_2 a_1 \oo{\frac{1}{2}+f_1'\oo{\hat{x}_1}}.
    \end{align*}

    Similarly, we have that
    \begin{align*}
        g\oo{\hat{x}}
        = 
        \begin{bmatrix}
            \hat{x}_2 a_1 \oo{\frac{1}{2} + f'\oo{\hat{x}_1}}\\
            \hat{x}_1 a_2 \oo{\frac{1}{2} + f'\oo{\hat{x}_2}}
        \end{bmatrix}
    \end{align*}
    Unlike its 1-dimensional counterpart, we do not $g\oo{\hat{x}}$ is uniformly bounded away (in conic sense/element-wise sense) from $(0,0)$. 

    Moreover, the curl
    \begin{align*}
        \int_{\epsilon_2 =1-\hat{x}_2}^1 d\epsilon_2 \times g\oo{\hat{x}}
        :=
        \frac{\partial}{\partial \hat{x}_2} \cc{g\oo{\hat{x}}}_1-\frac{\partial}{\partial \hat{x}_1} \cc{g\oo{\hat{x}}}_2
        =
        \frac{a_2-a_1}{2} + a_2 f_1'\oo{\hat{x}_2} - a_1 f_1'\oo{\hat{x}_1} \ne 0
    \end{align*}
    almost everywhere. Thus, there is no absolutely continuous potential $\hat{\hat{\varphi}}$ such that $\hat{g} = \nabla \hat{\hat{\varphi}}$.

\section{Downstream Work of~\cite{geng2025differentiable}}

    \paragraph{Propagated Error}

        The paper ``DEFT: Differentiable Automatic Test Pattern Generation", Li, et. al., 2025~\cite{li2025deft} cites the paper~\cite{geng2025differentiable} and employs the same method. Thus, it contains the same mathematical error.

    \paragraph{Papers that cite this paper}

        The papers~\cite{do2025hephaestus, zeng2025scalable} claim that this paper uses differentiable surrogate objective and performs gradient descent, which is an incorrect claim. \cite{do2025hephaestus} uses the paper as a benchmark, while \cite{zeng2025scalable} adapts the algorithm from the paper as part of ablation study.

        The paper~\cite{pu2025rome} says that this paper is the first to incorporate gradient information into the framework. Similarly, the paper~\cite{puhipo} says that this paper incorporate gradient information. This is correct, but it is important to note that such gradient information is used in a non-gradient descent manner. 

        The papers~\cite{pu2025coco, wang2025computing, genglamplace, salva2025denoising, baigraph, liu2025apollo, clarke2025learning, lihypergraph} just mention this paper in passing.
        
    \paragraph{ICLR 2025 Peer Review}

    The official reviews of~\cite{geng2025differentiable} is publicly available as an ICLR 2025 peer-review process.
    \begin{itemize}
        \item Reviewer y4sA says that the paper optimizes over a loss function.
        \item Reviewer FT2F says that the paper creates a Lagrangian loss and runs gradient of such loss.
        \item Reviewer LD6k and NhwL say that the paper uses gradient descent.
    \end{itemize}
  
    All four comments are incorrect: this paper proposes a loss function, but does not proceed to optimize over such loss function. In more than one dimension, the algorithm uses a non-gradient dynamic, so it cannot be rationalized as a gradient descent of any loss function.\footnote{as long as we assume that loss has to be absolutely continuous which is a common assumption.}

{\small
\bibliographystyle{alpha}
\bibliography{bibliography}}

\appendix
\section{How Theorem 3 of~\cite{geng2025differentiable} is Incorrect}
\label{appendix:theorem3isnotcorrect}

    The paper~\cite{geng2025differentiable} has the following theorem:
    \begin{incorrect_theorem}
    \label{theorem3}
        [Restated Theorem 3 of~\cite{geng2025differentiable}]
        There exists a positive scalar $\mu^* > 0$ such that, for any $\mu > \mu^*$,
        \begin{align*}
            \argmin_{\hat{x} \in [0,1]^d} c^\top \hat{x} + \mu \sum_{i=1}^m \mexpect{x \sim P_{\hat{x}}}{\max\oo{\left\{a_i^\top x-b_i, 0\right\}}}
            \subseteq
            \argmin_{\hat{x} \in [0,1]^d: \sum_{i=1}^m\mexpect{x \sim P_{\hat{x}}}{\max\oo{\left\{a_i^\top x-b_i, 0\right\}}} = 0} c^\top \hat{x}.
        \end{align*}
    \end{incorrect_theorem}
    
    It is easy to see that, for any $\mu > 0$, the set
    \begin{align*}
        \argmin_{\hat{x} \in [0,1]^d} c^\top \hat{x} + \mu \mexpect{x \sim P_{\hat{x}}}{\max\oo{\left\{a_i^\top x-b_i, 0\right\}}} \ne\emptyset.
    \end{align*}
    This happens because the objective function is continuous\footnote{and is, in fact, infinitely differentiable} and the search space $\cc{0,1}^d$ is closed and compact.

    However, the solution sets to the problem~\ref{eq:secondproblem}, which is $\argmin_{\hat{x} \in [0,1]^d: \mexpect{x \sim P_{\hat{x}}}{\max\oo{\left\{a_i^\top x-b_i, 0\right\}} =0}}$ can be empty. 

    Thus, we the theorem 3 of the paper is incorrect.

    \subsection{Incorrect Proof in the Paper}
    \label{subsection:incorrectproof3}

        The proof, which is in the appendix of the original paper, is not correct. It consider an arbitrary feasible solution to the problem~\ref{eq:thirdproblem}\footnote{Denoted as (P3) in the paper}\footnote{The feasible solution to the problem~\ref{eq:thirdproblem} is any element of $[0,1]^d$.} as $\hat{x}$, and claims that
        \begin{align*}
            \phi\oo{\hat{x}} := \sum_{i=1}^m = 0
        \end{align*}
        by feasibility. However, this only holds when $\hat{x}$ is a feasible solution to the problem~\ref{eq:secondproblem} and not the problem~\ref{eq:thirdproblem}.

        Thus, the proof requires the existence of a feasible solution to the problem~\ref{eq:secondproblem}.

\end{document}